\documentclass{amsart}
\usepackage{all2021}

\DeclareMathOperator{\partrank}{pr}
\DeclareMathOperator{\polyrank}{rk}
\DeclareMathOperator{\characteristic}{char}

\begin{document}

\title{On subtensors of high partition rank} 

\author{Jan Draisma}
\address{Mathematical Institute, University of Bern, Sidlerstrasse 5, 3012
Bern, Switzerland; and Department of Mathematics and Computer Science,
P.O. Box 513, 5600 MB, Eindhoven, the Netherlands}
\email{jan.draisma@unibe.ch}

\author{Thomas Karam}
\address{Mathematical Institute, University of Oxford,
Andrew Wiles Building, Radcliffe Observatory Quarter,
Woodstock Road, Oxford, OX2 6GG, United Kingdom}
\email{thomas.karam@maths.ox.ac.uk}

\thanks{JD is partially supported by Swiss National Science Foundation
(SNSF) project grant 200021\_191981 and by Vici grant 639.033.514 from the
Netherlands Organisation for Scientific Research (NWO). He thanks the
Institute for Advanced Study for excellent working conditions,
under which part of this project was carried out. TK is supported by the 
European Research Council (ERC) grant 883810. He thanks the Mathematical 
Institute, University of Oxford, for the very pleasant research environment, and thanks 
the CRM Montreal for organising the conference ``Tensors: Quantum Information, 
Complexity and Combinatorics" during which several of the main ideas of the present 
paper arose.}

\begin{abstract}
We prove that for every positive integer $d \ge 2$ there exist polynomial functions 
$F_d, G_d: \mathbb{N} \to \mathbb{N}$ such that for each positive integer $r$, 
every order-$d$ tensor $T$ over an arbitrary field and with partition rank at least 
$G_d(r)$ contains a $F_d(r) \times \cdots \times F_d(r)$ subtensor with partition rank 
at least $r$. We then deduce analogous results on the Schmidt rank 
of polynomials in zero or high characteristic.\end{abstract}

\maketitle

\tableofcontents

\section{Introduction and main result}

Let $K$ be a field, let $n_1, n_2, r$ be positive integers and let $A \in K^{n_1} 
\otimes K^{n_2}$ be a matrix
of rank $r$. Then $A$ contains a $r \times r$-submatrix of rank equal
to $r$. We want to prove a similar statement for order-$d$ tensors and
partition rank.

Let $n_1,\ldots,n_d$ be positive integers, and let $T \in K^{n_1} \otimes
\cdots \otimes K^{n_d}$ be a tensor. Then the {\em partition rank} of
$T$, introduced by Naslund \cite{Naslund20} and denoted $\partrank(T)$, is 
the smallest nonnegative integer $r$ such that $T$ can be
written as a sum of $r$ terms of the form $A \otimes B$ where $A
\in \bigotimes_{i \in I} K^{n_i}$ and $B \in \bigotimes_{i \not \in I}
K^{n_i}$ and $I$ is a proper subset of $[d]:=\{1,\ldots,d\}$ containing
$1$. We stress that $I$ may vary among the $r$ terms. For $d=2$, though,
$\{1\}$ is the only possible value for $I$, and it follows that then
$\partrank(T)$ is the matrix rank of $T$. Our paper will be
concerned mostly with $d \geq 3$. 

If $T$ is an order-$d$ tensor and $X_1, \dots, X_d$ are subsets of 
$[n_1], \dots, [n_d]$ respectively then we write $T[X_1, \dots, X_d]$ 
for the subtensor of $T$ obtained by restricting the entries of $T$ to the 
product $X_1 \times \dots \times X_d$. Set $r:=\partrank(T)$. We want to show that $T$ 
contains a subtensor of size some (not too large) function of $r$ and partition rank 
at least some (not too small) function of $r$.

The contrapositive says that if all small subtensors of $T$ have bounded
partition rank, then $T$ itself has bounded partition rank. This is what we
will prove.

\begin{thm}\label{thm: Bounds on subtensors}

Let $d \ge 2$ be a positive integer. There exist functions $F_d, G_d: \NN \to \NN$ 
such that if $r \ge 1$ is a positive integer and $T$ is an order-$d$ tensor over an arbitrary field such that every 
$F_d(r) \times \dots \times F_d(r)$ subtensor of $T$ has partition rank at most $r$, 
then $T$ has partition rank at most $G_d(r)$. Furthermore, we may take the bounds 
$F_d(r) \le (2^{d+3}r)^{2d}$ and $G_d(r) \le (2^{d+3}r)^{2d^2}$ for all $r$. 
\end{thm}

We remark that the bounds on the quantities $F_d(r)$ and $G_d(r)$ in
the theorem hold for any field $K$, though it is conceivable
that ``optimal'' functions $F_d$ and $G_d$ do depend on $K$. Furthermore, we 
believe that it is likely that Theorem \ref{thm: Bounds on subtensors} could still be 
true with both bounds $F_d(r)$ and $G_d(r)$ taken to grow linearly in $r$ for every 
fixed $d$, although we do not have a proof of that. In Section \ref{sec:ThreeTensors} 
we shall prove that we may take $F_3(r) = O(r^{3/2})$ 
and $G_3(r) = O(r^3)$.

The proof is inspired by an attempt to extend the following matrix argument
to higher-dimensional tensors. If $A$ is a matrix and $r$ is the largest nonnegative 
integer such that there exists an $r \times r$ submatrix $A[X,Y]$ of $A$
with rank $r$, then for every $x \in X^c$ and $y \in Y^c$ we have that
$\det A[X,Y] \neq 0$ but $\det A[X \cup \{x\}, Y \cup \{y\}] = 0$, so we can
express all coefficients $A(x,y)$ with $x \in
X^c$ and $y \in Y^c$ in the simple way \begin{equation} A(x,y) = A[\{x\}, Y] A[X,Y]^{-1}
A[X,\{y\}] \label{expression for matrices} \end{equation} in terms of the entries in the $r$
rows $A[\{x\},[n_2]]$ with $x \in X$ and the $r$ columns $A[[n_1],\{y\}]$
with $y \in Y$; this expression in turn this shows that the matrix $A[X^c,
Y^c]$ has rank at most $r$. Although the resulting bound is not optimal,
this argument allows us to deduce that $A$ must have rank at
most $3r$, since each of the $r$ rows $A[\{x\},[n_2]]$ with
$x \in X$ and each of the $r$ columns $A[[n_1],\{y\}]$ with
$y \in Y$ has rank at most $1$. We may hence try
to imitate this argument for the partition rank of order-$d$ tensors,
and ask the following question.

\begin{que} \label{bounded rank for the outside block} Let $d \ge 2$, $r \ge 1$ be positive integers. Does there exist a positive integer $C_d(r)$ 
satisfying the following ? If $T \in K^{n_1} \otimes
\cdots \otimes K^{n_d}$ is an order-$d$
tensor, $X_1, \dots, X_d$ are subsets of $[n_1], \dots, [n_d]$ respectively, each 
with size $r$, such that $\partrank T[X_1, \dots, X_d]= r$ and \[\partrank
T[X_1 \cup \{x_1\}, \dots, X_d \cup \{x_d\}] =r\] is satisfied for all
$x_1 \in [n_1] \setminus X_1, \dots, x_d \in [n_d] \setminus X_d$, then we have
\[\partrank T[[n_1] \setminus X_1, \dots, [n_d] \setminus X_d] \le C_d(r).\] \end{que}

As with the matrix argument, because every order-$(d-1)$
slice of $T$ has partition rank at most $1$, a positive
answer to Question \ref{bounded rank for the outside block} implies Theorem 
\ref{thm: Bounds on subtensors} with $F_d(r) = r$ and $G_d(r) = C_d(r) + dr$. In 
particular, it would suffice that $C_d$ be linear in $r$ for $G_d$ to be linear in $r$.

It is however not obvious to us that Question \ref{bounded rank
for the outside block} has a positive answer. Unlike in the case of
matrices, where all smallest-length rank decompositions of a full-rank matrix can be
deduced from one another via a change of basis, the set of partition
rank decompositions of a given full-rank tensor is richer in general,
already in the $d=3$ case; this makes it harder to obtain an analogue
of the expression \eqref{expression for matrices}.

In the absence of such a direct expression for $T(x_1, \dots, x_d)$, we
may ask for a weaker description: an equation of which $T(x_1, \dots,
x_d)$ is a solution, which leads the way to the
arguments that our proof will involve. A second difficulty which we have to circumvent 
is that we do not know that the set of tensors in $K^{n_1} \otimes
\cdots \otimes K^{n_d}$ with partition rank at most $r$ is
Zariski-closed in general: although this is true for algebraically
closed fields $K$ and $d=3$ (as shown by Sawin and Tao \cite{Tao16}),
this is likely false in general, even over algebraically closed $K$.
(Indeed, the corresponding notion of bounded strength for quartics is
not closed \cite{Ballico22}.) This will nonetheless
not interfere with our argument, as it will suffice for us to find a polynomial 
for which the zero-set merely contains the set of tensors in $K^{n_1} \otimes
\cdots \otimes K^{n_d}$ with partition rank at most $r$ rather than being equal to it. Indeed, our main stepping 
stone towards proving Theorem \ref{thm: Bounds on subtensors} will be the following statement.

\begin{thm} \label{thm:Bound}
Let $d \ge 2$, $r \ge 1$ be positive integers. Let $m \ge 1$ be a positive integer such that
there exist positive integers $n_1,\ldots,n_d$ and a nonzero polynomial $f$ of degree $m$
in $\CC[x_{i_1,\ldots,i_d} \mid i_j \in [n_j]]$ that vanishes on all
tensors in $\CC^{n_1} \otimes \cdots \otimes \CC^{n_d}$ of partition
rank at most $r$.

Then for any field $K$, any positive integers $n_1,\ldots,n_d \geq m$, and any tensor $T \in
K^{n_1} \otimes \cdots \otimes K^{n_d}$, if all $m \times \cdots \times
m$-subtensors of $T$ have partition rank at most $r$, then $T$ has
partition rank at most
\[ d(m-1) + \sum_{s=1}^{\lfloor d/2 \rfloor}
\binom{d}{s}(m-1)^{d-s}, \]
which for $d \geq 3$ is at most $m^d$. 
\end{thm}

We note that in the case of matrices, the determinant of the
top-left submatrix is such a polynomial, and we may hence
take $m=r+1$. This yields $4r$, almost recovering the bound
$3r$ discussed earlier.

In the case of high-characteristic fields, we may deduce from Theorem
\ref{thm: Bounds on subtensors} an analogue for polynomials. If $P$
is a homogeneous polynomial in several variables over a field $K$ and
with degree at least $2$, then we let $\polyrank P$ be the smallest positive
integer $k$ such that we may write \[ P = Q_1 R_1 + \dots + Q_k R_k
\] for some homogeneous polynomials $Q_i, R_i$ satisfying $\deg P_i,
\deg Q_i < \deg P$ and $\deg Q_i + \deg R_i = \deg P$ for each $i \in
[k]$. This notion of rank is known as the {\em Schmidt rank} or
{\em strength} of $P$.

For every subset $U \subset [n]$, we write $P[U]$ for the polynomial
in $K[x_u | u \in U]$ obtained by substituting in the polynomial $P$
the value $0$ for all variables in $[n] \setminus U$.

\begin{thm} \label{thm: Bounds on restricted polynomials} Let $K$ be a
field, let $d \geq 2$, $r \ge 1$ be positive integers and set $D:=\binom{d}{\lfloor d/2\rfloor} \leq
2^d$. Assume that $\characteristic K=0$ or $\characteristic K >d$.
If $P$ is a homogeneous polynomial in variables $x_1, \dots, x_n$ over
$K$ with $\deg P = d$ that satisfies $\polyrank P[U] \le r$ for every subset $U
\subset [n]$ with size at most $dF_d(r \cdot D)$, then 
\[\polyrank P \le G_d(r \cdot D) .\] \end{thm}

\subsection{Relations to the literature}
The main results and techniques present paper may be contrasted to those
of existing works in the literature. A general framework for studying
restriction-closed properties had been started in \cite{Draisma17},
but the techniques in that paper assume that the property is
Zariski-closed, which as we have explained does not appear to be the
case for the set of tensors in $K^{n_1} \otimes
\cdots \otimes K^{n_d}$ with partition rank at most $r$. The more recent work
\cite{Blatter22b} assumes that the field is finite.

High-rank subtensors were also studied in \cite{Karam22}, using very different 
arguments: Theorem 4.1 from that paper there is similar to (our) Theorem 
\ref{thm: Bounds on subtensors}, but assumes that the field is finite, an 
assumption which is heavily used in the proof, through a connection between 
the partition rank and the analytic rank; although the qualitative part of 
Theorem \ref{thm: Bounds on subtensors} follows as a special case of the first 
main theorem stated in the introduction of that paper, polynomial bounds in the 
functions $F_d$ and $G_d$ are a novelty of the present paper.

Again letting aside the matter of the bounds, one can deduce from a
universality theorem of Kazhdan and Ziegler \cite{Kazhdan20} a qualitative
version of Theorem \ref{thm: Bounds on restricted polynomials} where
the assumption is replaced by the requirement that every image $T'
\in K^{F_d(r)} \otimes \cdots \otimes K^{F_d(r)}$ of $T$ under any
$d$-tuple of linear transformations has partition rank at most $r$.
This condition is much stronger than our requirement on subtensors.
A general method for passing from a result about linear maps to a result
about subtensors, which involves fundamental results about finitely
generated $\mathbf{FI}$-modules, is described in \cite{Blatter22b}.

Let us finally mention a paper of Bri\"et and Castro-Silva
\cite{Briet22} on random restrictions of tensors and polynomials, which provides 
an additional motivation for the present line of work. Although none of their 
results imply ours or the other way around, they identify linear bounds in Theorem 
\ref{thm: Bounds on subtensors} and in Theorem \ref{thm: Bounds on restricted polynomials} 
as respectively providing a natural route to recover a random restriction 
theorem for tensors and for polynomials. Our proof of Theorem 
\ref{thm: Bounds on restricted polynomials} shows that linear bounds in 
Theorem \ref{thm: Bounds on subtensors} would suffice to obtain linear bounds 
in Theorem \ref{thm: Bounds on restricted polynomials}.

\subsection{Organisation of this paper}

In Section~\ref{sec:RankBound} we prove Theorem~\ref{thm:Bound},
in Section~\ref{sec:DegreeBound} we find a bound on $m$ in
Theorem~\ref{thm:Bound} in terms of $r$ and derive Theorem~\ref{thm:
Bounds on subtensors}. In Section~\ref{sec:ThreeTensors} we use
classical invariant theory to derive a slightly better bound on $m$
in the special case of $d=3$, and in Section \ref{sec:Polynomials} we deduce 
Theorem \ref{thm: Bounds on restricted polynomials} from Theorem 
\ref{thm: Bounds on subtensors}.

\section*{Acknowledgements}  We thank Jop Bri\"et for useful discussions, 
and Harm Derksen for suggesting to us the argument in Section \ref{sec:ThreeTensors}.

\section{Proof of Theorem~\ref{thm:Bound}}
\label{sec:RankBound}

We begin with the following elementary observation; this is inspired by
Snowden's alternative proof in \cite{Snowden21} of many of the results
in \cite{Bik21} by working with equations of weight
$(1,\ldots,1)$. 

\begin{lm}
We may assume that the polynomial $f$ in Theorem \ref{thm:Bound} lies in 
\[ \CC[x_{i_1,\ldots,i_d} \mid i_j \in [m]], \]
has coefficients in $\ZZ$ with gcd $1$, and has weight
$(1^m,\ldots,1^m)$ for the torus $((\CC^*)^m)^d$.
\end{lm}

\begin{proof}
Since tensors of partition rank at most $r$ are the image of a map defined
over $\ZZ$, we may assume that $f$ has integer coefficients. 
Since they are preserved by coordinate
scalings, we may further assume that $f$ is a weight vector, i.e., $f$
gets scaled by $(t_1^{\alpha_1},\ldots,t_d^{\alpha_d})$, for certain
$\alpha_i \in \ZZ_{\geq 0}^{n_i}$, when the tensor gets
acted upon by
$(\diag(t_1),\ldots,\diag(t_d)) \in \prod_{i=1}^d \GL_{n_i}(\CC)$. Now
if the $j$-th entry of $\alpha_i$ is strictly greater than $1$, then acting with the Lie
algebra element $E_{j,n_i+1}$ on $f$ we get another polynomial
that vanishes on tensors of partition rank $r$ and which has weight
$(\alpha_1,\ldots,\alpha_i',\ldots,\alpha_d)$, where
\[
\alpha_i'=(\alpha_{i1},\ldots,\alpha_{ij}-1,\ldots,\alpha_{in_i},1)
\in \ZZ_{\geq 0}^{n_i+1}.
\]
We replace $n_i$ by $n_i+1$ and $\alpha_i$ by $\alpha_i'$.
Continue in this manner until all $\alpha_i$ only have $1$s and $0$s as
entries. The $0$s correspond to slices of variables that do not occur
in $f$. After removing these, $f$ has weight $(1^{n_1},\ldots,1^{n_d})$,
where we note that the $n_i$ may have changed. Each
variable has weight $(e_{j_1},\ldots,e_{j_d})$ for some $j_i
\in [n_i]$, and it follows that all $n_i$ are equal to a
common number $m$. Finally, divide the resulting $f$ by an
integer to ensure that the coefficients have gcd $1$. 
\end{proof}

Now consider the image of $f$ in $K[x_{i_1,\ldots,i_d} \mid i_1,\ldots,i_d
\in [m]]$. This is nonzero, since the coefficients of $f$ have gcd $1$,
and it is still a weight vector of weight $(1^m,\ldots,1^m)$. From now on,
we write $f$ for the image. Note that $f$ vanishes on tensors in $K^{m}
\otimes \ldots \otimes K^{m}$ of partition rank at most $r$.

Since $f$ has weight $(1^m,\ldots,1^m)$, after applying
permutations in the $d$ directions, we can write
\[ f=:h_0=x_{m,\ldots,m} h_1 + r_1 \]
where $h_1$ is a nonzero weight vector of weight
$(1^{m-1}0,\ldots,1^{m-1}0)$ and where $r_1$ is a weight vector of
weight $(1^m,\ldots,1^m)$ that does not involve $x_{m,\ldots,m}$. Note that
the variables in $h_1$ have all indices at most $m-1$, and all variables
in $r_1$ have at least one index at most $m-1$.

Similarly, after further permutations on
the first $(m-1)$ indices, for $k=1,\ldots,m-1$ we have 
\[ h_k=x_{m-k,\ldots,m-k} h_{k+1} + r_{k+1} \]
where $h_{k+1}$ is a nonzero weight vector of weight
$(1^{m-k-1}0^{k+1},\ldots,1^{m-k-1}0^{k+1})$ and $r_{k+1}$ a weight
vector of weight $(1^{m-k}0^k,\ldots,1^{m-k}0^k)$ that does not involve
$x_{m-k,\ldots,m-k}$. Note that $h_m$ is a nonzero constant and that
$r_m=0$. Furthermore, for each $i=1,\ldots,d$, every term of $r_{k+1}$
contains precisely one variable that has an index $m-k$ on position
$i$. These $d$ indices $m-k$ are distributed over at least two and at
most $d$ of the variables in the term. So $r$ is a linear combination
of terms of the following form (illustrated for $d=4$):
\[ x_{m-k,i_2,i_3,i_4} \cdot x_{j_1,m-k,m-k,j_4} \cdot 
x_{l_1,l_2,l_3,m-k} \]
where $i_2,\ldots,l_3 \in [m-k-1]$ and where the coefficients are
polynomials in the variables all of whose indices are in $[m-k-1]$.

\begin{prop} \label{prop:Bound}
Let $n_1,\ldots,n_d \geq m$ be positive integers and let $T \in K^{n_1} \otimes \cdots \otimes
K^{n_d}$. Suppose that, for some $k=0,\ldots,m-1$, the whole $G:=\prod_i
\Sym([n_i])$-orbit of $h_k$ vanishes at $T$, but not the whole $G$-orbit
of $h_{k+1}$ vanishes at $T$. Then $T$ has partition rank at most
\begin{equation}  \label{eq:Bound}
d(m-k-1) + \sum_{s=1}^{\lfloor d/2 \rfloor} \binom{d}{s}
(m-k-1)^{d-s}. 
\end{equation}
\end{prop}

\begin{proof}
Without loss of generality, we may assume that $h_{k+1}(T)$ is nonzero
and the whole $G$-orbit of $h_k$ is zero on $T$. We then have
\begin{equation} \label{eq:T} t_{m-k,\ldots,m-k} =
-r_{k+1}(T)/h_{k+1}(T). \end{equation}
For instance, if $d=4$, then $t_{m-k,\ldots,m-k}$ is a linear combination
of terms such as 
\[ t_{m-k,i_2,i_3,i_4} \cdot t_{j_1,m-k,m-k,j_4} \cdot  
t_{l_1,l_2,l_3,m-k} \]
where the coefficients only depend on the subtensor
$T[[m-k-1],\ldots,[m-k-1]]$. Since the whole $G$-orbit of
$h_k$ vanishes on $T$, we may apply to \eqref{eq:T} arbitrary elements
of the subgroup $G':=\prod_{i=1}^d \Sym([n_i] \setminus
[m-k-1])$ of $G$. Note that this fixes all indices up to $m-k-1$. As a consequence, 
we find that the subtensor 
\[ T[[n_1] \setminus [m-k-1],\ldots,[n_d] \setminus [m-k-1]]
\]
of $T$ admits a decomposition as a sum of tensor products of
tensors in which every term is
divisible by some tensor like (for $d=4$)
\[ (t_{j_1,a_2,a_3,j_4})_{a_2 \in [n_2] \setminus [m-k-1], a_3
\in [n_3] \setminus [m-k-1]}. \]
for some choice of $(j_1,j_4) \in [m-k-1]^2$. In each term
of this decomposition, there is at least one factor which is a tensor
of order at most $\lfloor d/2 \rfloor$.
The number of these tensors equals 
\[ \sum_{s=1}^{\lfloor d/2 \rfloor} \binom{d}{s}
(m-k-1)^{d-s}. \]
where $\binom{d}{s}$ counts the choices of positions $i$
where one puts the indices varying in $[n_i] \setminus [m-k-1]$ (positions $2,3$
in the example) and the factor $(m-k-1)^{d-s}$ counts the number of
choices for the remaining indices ($j_1,j_4$ in the example).

Finally, the remainder of $T$ admits a {\em slice rank} decomposition
where each term is divisible by some standard basis vector in $K^{m-k-1}$,
in one of the $d$ factors. This is accounted for by the first term in
\eqref{eq:Bound}.
\end{proof}

\begin{proof}[Proof of Theorem~\ref{thm:Bound}.]
By construction of $f=h_0$, its entire $G$-orbit vanishes on the given
tensor $T$. On the other hand, $h_m$ is a nonzero constant. Hence there exists
a $k \in \{0,\ldots,m-1\}$ such that the entire $G$-orbit of $h_k$
vanishes at $T$ but not the entire $G'$-orbit of $h_{k+1}$ vanishes
at $T$. Hence Proposition~\ref{prop:Bound} applies with this
$k$.  Now the bound
in Theorem~\ref{thm:Bound} follows from the bound in \eqref{eq:Bound}
by taking the worst $k$, namely, $k=0$. 
\end{proof}

\section{A degree bound for $f$ and proof of
Theorem~\ref{thm: Bounds on subtensors}}
\label{sec:DegreeBound}

\begin{thm} \label{thm:Boundm}
Let $d \ge 2$ be a positive integer. Then for every positive integer $r \ge 1$ and for $n=8r$, there exists a polynomial $f$ with degree at most 
$m \le (2^{d+3}r)^{2d}$ in $\CC[x_{i_1,\ldots,i_d} \mid i_j \in [n]]$ that 
vanishes on all tensors in $\CC^{n} \otimes \cdots \otimes \CC^{n}$ ($d$ factors) 
with partition rank at most $r$.
\end{thm}

\begin{proof}
We write $X_{n,r}$ for the set of order-$d$ tensors in $\CC^n
\otimes \cdots \otimes \CC^n$  with partition
rank at most $r$. The set $X_{n,r}$ is contained in the set $X_{n,r}'$ of order-$d$ 
tensors $T$ that have a partition
rank decomposition of the type
\begin{equation} T = \sum_{I \in \mathcal{P}([d]) \setminus \{\emptyset, [d]\}} \sum_{i=1}^{r}
A_{I,i} \otimes B_{I^c,i} 
\label{partition rank decomposition}
\end{equation}
for certain tensors $A_{I,i} \in \bigotimes_{j \in I} \CC^n$
and $B_{I^c,i} \in \bigotimes_{j \in I^c} \CC^n$.

Note that we could take the summation over half of these $I$, but for simplicity 
we will not do so here. From now on, the range of $I$ in summations, indexations etc. 
will always be taken to be $\mathcal{P}([d]) \setminus \{\emptyset, [d]\}$ unless indicated 
otherwise, where $\mathcal{P}([d])$ is the power set of $[d]$.

Let $\pi_{r} =  \prod_{I} (\CC^{n^I} \times \CC^{n^{I^c}})^{r}$ be the parameter space 
for decompositions as in \eqref{partition rank decomposition}.

We let $P_{2m}(\pi_{r})$ be the vector space of
homogeneous polynomials of degree $2m$ in the variables $A_{I,i}, B_{I,i}$
and let $P_{m}(\CC^{[n]^d})$ be the linear space of
homogeneous polynomials of degree $m$ in the $n^d$ entries of tensors of $\CC^{[n]^d}$.
Letting $\phi: \pi_{r} \to \CC^{[n]^d}$ be the
parametrisation defined by taking \[\phi(A_{I,i}, B_{I,i}:
I, 1 \le i \le r)\] to be the right-hand side of
\eqref{partition rank decomposition}, the pull-back
$\phi^\#: P_{m}(\CC^{[n]^d}) \to P_{2m}(\pi_{r})$
with $\phi^\#(f) = f \circ \phi$ is well-defined.

If $f$ in an element of $\ker \phi^\#$, then $f$ vanishes on
the image $X_{n,r}'$ of $\phi$, so it now suffices to
check that $\ker \phi^\# \neq \{0\}$. In turn, to show this, it suffices to 
show \[ \dim P_{2m}(\pi_{r}) < \dim P_{m}(\CC^{[n]^d}).\]

Defining $S_{n,r} = \sum_I r (n^{|I|} + n^{d-|I|}) =
\dim \pi_{r}$, we can write \begin{align*} \dim P_{2m}(\pi_{r}) & = 
\binom{2m + S_{n,r} - 1}{2m}\\ & = \binom{2m + S_{n,r} - 1}{S_{n,r} - 1} \\ 
& \le (2m + S_{n,r})^{S_{n,r}-1} / (S_{n,r}-1)!.\end{align*}

Meanwhile \begin{align*} \dim P_{m}(\CC^{[n]^d}) & = \binom{m+n^d-1}{m} \\ 
& = \binom{m+n^d-1}{n^d-1} \\ & \ge m^{n^d-1} / (n^d-1)!.\end{align*}

It hence suffices to show \begin{equation} m^{n^d-1} / (2m + S_{n,r})^{S_{n,r}-1} 
> (n^d-1)! / (S_{n,r}-1)! \label{compared quantities}. \end{equation} 

For every $1 \le |I| \le d-1$ we have $n^{|I|} + n^{d-|I|} \le 2
n^{d-1}$, so $S_{n,r} \le 2^{d+1}rn^{d-1}$. Assuming that $m \ge n^d/4 \ge 3$
and $n=2^{d+3}r$, the left-hand side of \eqref{compared quantities} is therefore at least 
\begin{align*} m^{n^d-1} / (2m+2^{d+1}rn^{d-1})^{2^{d+1}rn^{d-1}-1} & \ge 
m^{n^d-1} / (2m+n^d/4)^{n^d/4-1} \\ & \ge m^{n^d-1} / (3m)^{n^d/4-1} \\ & \ge m^{n^d/2}.\end{align*}

Meanwhile, the right-hand side of \eqref{compared quantities} is at
most its numerator, and hence at most $(n^d)^{n^d}$. 
Therefore, for \eqref{compared quantities} to hold, it suffices that 
\[m^{n^d/2} > (n^d)^{n^d},\] 
which simplifies to $m \ge n^{2d}$. Since $n = 2^{d+3}r$, it suffices that 
$m \ge (2^{d+3}r)^{2d}$ for there to exist a polynomial $f$ which is zero on 
$X_{n,r}'$ and hence on $X_{n,r}$.

\end{proof}

\begin{proof}[Proof of Theorem~\ref{thm: Bounds on subtensors}]
Take $F_d(r)=m:=(2^{d+3}r)^{2d}$. By Theorem~\ref{thm:Boundm}, there exists
a nonzero polynomial $f$ of degree $m$ that vanishes on order-$d$ tensors
of partition rank $r$, and by Theorem~\ref{thm:Bound}, any $n_1 \times
\cdots \times n_d$-tensor all of whose 
$m \times \cdots \times m$-subtensors have partition rank at most $r$,
has itself partition rank at most $m^d=(2^{d+3}r)^{2d^2}=:G_d(r)$, as desired. 
Here we used $d \geq
3$ in the last step, but as discussed in the beginning of the paper,
for $d=2$ even much better bounds work.

Finally, note that if some $n_i$ happens to be smaller than $m$, then
the tensor has partition rank at most $m < G_d(r)$.
\end{proof}

\section{Order-$3$ tensors and invariant theory}
\label{sec:ThreeTensors}

In this section, we focus on $d=3$ and follow a construction suggested to
us by Harm Derksen. In this case, partition rank equals slice rank. The
following is well-known, but we include a quick proof. 

\begin{lm} \label{lm:Nullcone}
The tensors $T \in \CC^n \otimes \CC^n \otimes \CC^n$ of slice rank
strictly less than $n$ are contained in the nullcone for the action of
the group $G:=\SL_n \times \SL_n \times \SL_n$.
\end{lm}

Here the nullcone is the set of all vectors on which all
$G$-invariant polynomials vanish.

\begin{proof}
For such a tensor there exist linear subspaces $V_1,V_2,V_3 \subseteq \CC^n$
with $\dim(V_1)+\dim(V_2)+\dim(V_3) < n$ such that
\[ T \in V_1 \otimes \CC^n \otimes \CC^n 
+ \CC^n \otimes V_2 \otimes \CC^n
+ \CC^n \otimes \CC^n \otimes V_3. \]
After linear coordinate changes in the individual tensor factors, we
may assume that $V_i$ is spanned by the first $n_i$ basis vectors. Now
consider the triple $\lambda:=(\lambda_1,\lambda_2,\lambda_3)$ of
$1$-parameter subgroups in $\SL_n$ defined by
\[
\lambda_i(t)=\diag(t^{(n-n_i)},\ldots,t^{(n-n_i)},t^{-n_i},\ldots,t^{-n_i})
\] 
where there are $n_i$ copies of the $t^{(n-n_i)}$ and $n-n_i$ copies of
$t^{n_i}$, so that $\det(\lambda_i(t))=1$ as desired. Then one sees that
$\lambda(t)$ acts by some power $t^a$ on each standard basis vector in
$\CC^n \otimes \CC^n \otimes \CC^n$, and that $a \geq n-n_1-n_2-n_3>0$
for all basis vectors that have a nonzero coefficient in $T$. Hence
$\lambda(t) T \to 0$ for $t \to 0$ and all $G$-invariant polynomials
vanish on $T$.
\end{proof}

For the following result we refer to \cite{Buergisser17},
where this invariant is called $F_k$. 

\begin{prop} \label{prop:Kronecker}
If $n=k^2$, then there exists a nonzero, homogeneous, $G$-invariant
polynomial on $\CC^n \otimes \CC^n \otimes \CC^n$ of degree
$k^3$. \hfill $\square$
\end{prop}

\begin{cor}
For $d=3$, in Theorem~\ref{thm:Bound}, we can take $m=O(r \sqrt{r})$
for $r \to \infty$.
\end{cor}

\begin{proof}
Let $k$ be the smallest positive integer satisfying $k^2 \geq r+1$ and set
$n:=k^2$.  By Proposition~\ref{prop:Kronecker}, there exists an nonzero
$G$-invariant polynomial $f$ of degree $k^3$ on $\CC^n \otimes \CC^n
\otimes \CC^n$. By Lemma~\ref{lm:Nullcone}, this polynomial vanishes
on tensors of slice rank at most $n-1$, hence in particular in tensors of
slice rank at most $r$. Clearly, $\deg(f)=O(r \sqrt{r})$. 
\end{proof}

\section{Deduction of the restriction result for the rank of polynomials}
\label{sec:Polynomials}

In this section we deduce Theorem \ref{thm: Bounds on restricted
polynomials} from Theorem \ref{thm: Bounds on subtensors}.
We begin by establishing that we can deduce bounds on the partition
rank of a tensor from bounds on the Schmidt rank of a polynomial and
conversely.  This correspondence is well known from the work of Kazhdan
and Ziegler \cite{Kazhdan20}; we include it here for convenience of the
reader. As in the statement of Theorem \ref{thm: Bounds on restricted
polynomials} we assume that $\characteristic K = 0$ or $\characteristic K > d$.
Recall that the space of homogeneous polynomials of degree $d$ in 
$x_1,\ldots,x_n$ has a natural isomorphism to the symmetric power
$S^d K^n$, and consider the natural linear map determined by 
\[ \phi:K^n \otimes \cdots \otimes K^n \to S^d K^n,\  v_1 \otimes \cdots
\otimes v_d \mapsto v_1 \cdots v_d. \]
It clearly maps any tensor of partition rank at most $1$ to a polynomial of
Schmidt rank at most $1$, and hence, by linearity, a tensor of
partition rank at most $r$ to a polynomial of Schmidt rank at most
$r$. Conversely, we have a linear map determined by
\[ \psi:S^d K^n \to K^n \otimes \cdots \otimes K^n,\ 
v_1\cdots v_d \mapsto \sum_{\pi \in S_d} v_{\pi(1)} \otimes \cdots
\otimes v_{\pi(d)}, \]
which is a linear isomorphism to the space of {\em symmetric} tensors
in $K^n \otimes \cdots \otimes K^n$ with inverse $\phi/d!$ restricted to
that space of symmetric tensors. This maps a polynomial of Schmidt
rank at most $1$, given as $Q \cdot R$ with $Q$ of degree $e$ and $R$
of degree $d-e$, to a tensor of partition rank at most
$\binom{d}{e} \leq 
\binom{d}{\lfloor d/2 \rfloor}=:D$. Again by linearity, this map sends a
polynomial of Schmidt rank at most $r$ to a tensor of partition rank at
most $r \cdot D$.

\begin{proof}[Proof of Theorem~\ref{thm: Bounds on restricted polynomials}]
Suppose that $P$ is a homogeneous polynomial of degree $d$ in
$x_1,\ldots,x_n$ and $\polyrank(P[U]) \leq r$ for all $U \subseteq [n]$
of size at most $d F_d(r D)$. Then $\psi(P)=:T$ is a (symmetric) tensor.
Consider any $d$-tuple $U_1,\ldots,U_d$ of subsets of $[n]$, each of size
at most $F_d(r D)$, and take $U:=U_1 \cup \cdots \cup U_d$, a set of size at
most $d F_d(rD)$. Then by our assumption $\phi(T[U,\ldots,U])=d! P[U]$
has Schmidt rank at most $r$. It follows that
$T[U,\ldots,U]=\psi(P[U])$ has partition rank at most $r D$, and hence {\em a
fortiori} so does $T[U_1,\ldots,U_d]$. We conclude that $T$ itself has
partition rank at most $G_d(rD)$, and therefore $P=\phi(T)/d!$ has Schmidt rank at
most $G_d(rD)$.
\end{proof}

\bibliographystyle{alpha}
\bibliography{diffeq,draismapreprint,draismajournal}

\end{document}